\documentclass{amsart}
\usepackage{amsmath,amssymb, amsthm}
\usepackage{enumerate}
\usepackage{hyperref}
\usepackage{dsfont}
\usepackage{mathtools}
\usepackage{float}

\newtheorem{theorem}{Theorem}
\newtheorem{lemma}[theorem]{Lemma}

\newtheorem{observation}[theorem]{Observation}
\newtheorem{fact}[theorem]{Fact}
\theoremstyle{definition}
\newtheorem{definition}[theorem]{Definition}
\newtheorem{remark}[theorem]{Remark}
\newtheorem{example}[theorem]{Example}
\newtheorem{claim}{Claim}

\newtheorem{theoremOne}{Theorem}

\DeclareMathOperator{\ran}{ran}
\DeclareMathOperator{\dm}{d}
\DeclareMathOperator{\cl}{cl}
\DeclareMathOperator{\acl}{acl}
\DeclareMathOperator{\SRM}{SRM}

\newcommand{\set}[1]{\{#1\}}
\newcommand{\setcol}[2]{\{#1 : #2\}}

\renewcommand{\phi}{\varphi}
\newcommand{\M}{\mathcal{M}}
\newcommand{\N}{\mathcal{N}}
\newcommand{\B}{\mathcal{B}}
\newcommand{\parity}{\equiv}
\newcommand{\mt}{\,\text{(mod 2)}}
\newcommand{\ILD}{{\text{I}\Lambda \text{D}}}

\renewcommand{\setminus}{\smallsetminus}

\title{Recursive spectra of flat strongly minimal theories}
\author{Uri Andrews, Omer Mermelstein}
\address{Department of Mathematics, University of Wisconsin--Madison, 480 Lincoln Dr., Madison, WI 35706, USA}
\email{andrews@math.wisc.edu, omer@math.wisc.edu}
\thanks{The first author was partially supported by NSF grant DMS-1600228.}


\begin{document}
\maketitle

\section{Introduction}

\begin{abstract}
We show that for a model complete strongly minimal theory whose pregeometry is flat, the recursive spectrum (SRM($T$)) is either of the form $[0,\alpha)$ for $\alpha\in \omega+2$ or $[0,n]\cup\set{\omega}$ for $n\in \omega$, or $\{\omega\}$, or contained in $\set{0,1,2}$.

Combined with previous results, this leaves precisely 4 sets for which it is not yet determined whether each is the spectrum of a model complete strongly minimal theory with a flat pregeometry.
\end{abstract}

Recursive model theory examines what information is required to recursively present models of a given theory. Of particular interest is the case where the theory itself poses no obstacles to computing its models. 
This is exemplified by the case where a theory $T$ has both recursive and non-recursive models. The spectrum problem for strongly minimal theories focuses on the clearest case of this phenomenon. From this point onward, we only consider complete theories. The Baldwin-Lachlan theorem \cite{BL} shows that the countable models of a strongly minimal non-$\aleph_0$-categorical theory $T$ are characterized by dimension and thus $T$ has a countable chain of countable models: $\M_0\preceq \M_1\preceq \cdots \preceq \M_\omega$. We let the spectrum of recursive models of $T$, or $\SRM(T)$, be the set of $k\in \omega+1$ so that $\M_k$ has a recursive presentation. The spectrum problem asks to characterize the sets $S\subseteq \omega+1$ which occur as $\SRM(T)$. The spectrum problem has been a focus in recursive model theory since Goncharov \cite{Go78} first showed the existence of a non-trivial spectrum. 

It has proven useful \cite{GHLLM,Lobachevski,AM} to examine the spectrum problem in light of the pregeometry of the theory. Zilber famously conjectured a trichotomy for the possible pregeometries of strongly minimal theories: That the pregeometry is either disintegrated, locally modular, or field-like. Hrushovski \cite{Udi} produced a counterexample, thus introducing a new type of pregeometry: The flat pregeometries. In particular, flat pregeometries are non-disintegrated and preclude the theory from interpreting a group, so flat pregeometries do not fit into Zilber's conjectured trichotomy. Hrushovski \cite{HrushovskiSecond} also presented a construction to fuse together strongly minimal theories which, when applied to theories with non-flat geometries, yields a new theory whose geometry does not satisfy the Zilber conjecture and is also not flat.

Partial results have been obtained for the spectrum problem in the case of each of the trichotomous pregeometries. For example, Andrews and Medvedev \cite{AM} show that the only spectra of disintegrated strongly minimal theories with finite signatures are $\emptyset,\{0\},\omega+1$. They also show that if $T$ is a locally modular strongly minimal theory with a finite signature which expands a group, then $\SRM(T)$ is among $\emptyset,\{0\},\omega+1$. They also note that field-like strongly minimal theories with finite signatures expanding a field have $\SRM$ of $\omega+1$.

So far, we have very little information about the spectra of flat strongly minimal theories, i.e., strongly minimal theories with flat pregeometries. Each of these other cases: Disintegrated theories, locally modular theories expanding a group, or field-like theories expanding a field admit some level of quantifier-elimination after naming parameters. This is not true for flat strongly minimal theories (e.g., \cite{TheoryArithmetic}). In this paper, we examine the spectrum problem for \emph{model complete} flat strongly minimal theories. This is a wide enough class to encompass all the existing constructions of spectra of a flat strongly minimal theory.

In particular, we show:

\begin{theorem}
	If $T$ is a model complete flat strongly minimal theory, then $\SRM(T)$ is of one of the following forms:
	
	\begin{itemize}
		\item $[0,\alpha)$ for $\alpha\in \omega+2$, $[0,n]\cup \set{\omega}$ for $n\in \omega$, $\{\omega\}$, 
		\item $\{1\}$, $\{2\}$, $\{0,2\}$, $\set{1,2}$, 
	\end{itemize}
	
	Further, the sets in the first line are in fact spectra of
	model complete flat strongly minimal theories by \cite{A0n} and \cite{AndrewsMermelstein}. We do not know whether any of the 4 sets (not schema) in the second line are spectra of model complete flat strongly minimal theories.
\end{theorem}

We note that the proof of this result requires both some recursion theoretic arguments (as in section \ref{sec:Strategy}) and some purely geometric arguments (as in sections \ref{sec:PPS} and \ref{sec:relatingcircuitsizetomodeldimensions}).

\section{Background}
A \emph{combinatorial pregeometry} $G$ is a set $X$ equipped with an operator $\cl:\mathcal{P}(X)\to \mathcal{P}(X)$ satisfying, for every $A\subseteq B\subseteq X$, $a,b\in X$:
\begin{itemize}
\item
$A\subseteq \cl(A) \subseteq \cl(B) = \cl(\cl(B))$ \hfill (Closure operator)
\item
if $a\in \cl(A\cup\set{b})\setminus \cl(A)$, then $b\in \cl(A\cup\set{a})$ \hfill(Exchange Principle)
\item
$\cl(A) = \bigcup\setcol{\cl(A_0)}{A_0\subseteq A, A_0 \text{ finite}}$ \hfill (Finite character)
\end{itemize}
We say that $A$ is closed if $A=\cl(A)$.
We say that ``$a$ depends on $A$'' when $a\in \cl(A)$. We say that a set of points $A$ is \emph{independent over} $B$ if $a\notin \cl(B\cup (A\setminus\set{a}))$ for every $a\in A$. We say $A$ is \emph{independent} if it is independent over the empty set. A \emph{basis} for a set $A$ is a minimal $Y\subseteq A$ such that $A\subseteq \cl(Y)$. The \emph{dimension} of a set $A$, denoted $\dm(A)$, is the cardinality of any basis of $A$ --- it follows from the exchange principle that all bases have the same cardinality.

A pregeometry is \emph{disintegrated} if whenever $b$ depends on $A$, then $b$ depends on some single element $a_0\in A$, i.e., $\cl(A) = \bigcup_{a_0\in A} \cl(a_0)$.

For $\N$ a strongly minimal structure, we define $G_\N$ to be the pregeometry given by the algebraic closure operator $\acl_{\N}$. When the structure $\N$ is clear from context, we drop the subscript. The pregeometry of a strongly minimal theory $T$ is $G_\M$, where $\M$ is the countable saturated model of $T$. In this paper, we will be interested in theories whose pregeometry is \emph{flat}:

\begin{definition}
In the context of an ambient pregeometry $G$, let $\Sigma = \set{E_1,\dots,E_k}$ be a collection of finite dimensional closed subsets of $G$.
We denote
\[
\Delta(\Sigma) = \sum_{\emptyset\neq s\subseteq \set{1,\dots, k}} (-1)^{|s|+1}\dm\left(\bigcap_{i\in s} E_i\right).
\]
This is the result of applying the inclusion-exclusion principle to the elements of $\Sigma$, where cardinality is replaced with dimension.

A pregeometry $G$ is \emph{flat} if it is non-disintegrated and whenever $\Sigma$ is a finite collection of finite dimensional closed sets in $G$, then $\Delta(\Sigma) \geq \dm(\bigcup \Sigma)$.
\end{definition}

\begin{example}
	Any strongly minimal theory which expands a group, such as $\text{Th}(\mathbb{Q},+)$ or $\text{ACF}_p$, cannot be flat. To see this, fix a generic triple $a,b,c$ in the group. Let $S_1 = \cl(a,b)$, $S_2=\cl(b,c)$, and $S_3=\cl(a,b\cdot c)$, and $S_4=\cl(a\cdot b,c)$. Then any two $S_i$'s intersect in a set of dimension 1, and any three have empty intersection, so $\sum_{\emptyset\neq s\subseteq \set{1,\dots, k}} (-1)^{|s|+1}\dm\left(\bigcap_{i\in s} E_i\right)= -6 +4\cdot 2 = 2$ and yet the dimension of $\bigcup_i S_i$ is 3.
\end{example}

Note that we include in the definition of flatness that the pregeometry is non-disintegrated. This choice is not universally made in the literature, but it allows us to concisely focus on the collection of pregeometries we wish to consider.

\begin{definition}
	A pregeometry $G$ is called \emph{homogeneous} if for any set $A$, for any elements $a,b\in G\setminus\cl(A)$, there is an automorphism $\sigma$ of $G$ over $A$ with $\sigma(a) = b$. In particular, for any two independent tuples $\bar{a}$, $\bar{b}$ of the same length, there is an automorphism $\sigma$ of $G$ with $\sigma(\bar{a}) = \bar{b}$.
\end{definition}

The pregeometry of a model of a strongly minimal theory $T$ is homogeneous, because every model of $T$ is homogeneous and all independent tuples of length $k$ have the same type --- the unique generic $k$-type.

The following useful fact follows directly from the Tarski-Vaught test for being an elementary substructure.
\begin{fact}\label{TVFact}
	If $\N$ is strongly minimal and $X\subseteq N$ is algebraically closed and infinite, then $X\preceq \N$. It follows that for cardinals $\alpha<\beta$, if a strongly minimal theory $T$ has a model of dimension $\alpha $, then $T$ has a model of dimension $\beta$.
\end{fact}

In examining the spectrum problem, there is a slight disconnect between the indices used in the definition of the spectrum of $T$ and the dimensions of the models of $T$. It follows from Fact \ref{TVFact} that if $p$ is the dimension of the prime model of $T$, then $\M_k$ (the $k$th model in the countable chain of models of $T$) has dimension $p+k$. Thus, if we were to define spectra in terms of dimensions, that notion would capture essentially the same information, but it would be far easier to witness sets such as $[10,\omega]$ as a spectrum, since we could just make the prime model have dimension 10.

\begin{fact}\label{fact:Carousel}
	Let $\bar{a}$ be any tuple in a strongly minimal structure, and let $b_1\ldots, b_n$ be a generic tuple over $\bar{a}$. For each $1\leq i\leq n$, let $\bar{c}_i=\bar{a}, b_1,\ldots b_{i-1},b_{i+1},\ldots b_n$. Then $\bigcap_{i=1}^n \acl(\bar{c}_i)=\acl(\bar{a})$.
\end{fact}
\begin{proof}
	We prove this by induction on $n$. For $n=1$, this is evident. For $n+1$, the inductive hypothesis shows $\bigcap_{i=1}^n \acl(\bar{c}_i)=\acl(\bar{a},b_n)$. By exchange, any element $x\in \acl(\bar{a},b_n)\smallsetminus \acl(\bar{a})$
	 has the property that $b_n\in \acl(\bar{a},x)$. Therefore, such an $x$ cannot be in $\acl(\bar{c}_n)$. So, $\bigcap_{i=1}^{n+1} \acl(\bar{c}_i)=\acl(\bar{a})$.
\end{proof}

We use the following notation throughout: If $\N$ is a structure, $\bar{a}\in \N$, and $\phi(\bar{x},\bar{y})$ is a formula, then $\phi(\N,\bar{a})=\setcol{\bar{b}}{\N\models \phi(\bar{b},\bar{a})}$.

All signatures are assumed to be recursive. Furthermore, since it makes no difference for recursive presentations, we assume that all signatures are relational.
\section{Our main strategy}\label{sec:Strategy}

The following Lemma is our main tool in showing that if the $k+1$-dimensional model is recursive, then the $k$-dimensional model is recursive.

\begin{lemma}\label{GoingDownTrick}
	Suppose that $\N$ is a recursive strongly minimal structure and ${\M\preceq \N}$ is a $\Delta^0_2$ subset of $\N$. Further suppose that $A\subseteq \M$ is any infinite $\Sigma^0_1$ set. Then $\M$ has a recursive copy.
\end{lemma}
\begin{proof}
	Let $M_s$ be the elements of $\N$ which are in $\M$ according to stage $s$ of the $\Delta^0_2$-approximation to $\M$. We can alter the approximation to ensure that if $x\in A_s$, the elements enumerated into $A$ by stage $s$, then $x\in M_s$.
	We build $\B$ in stages, so that at stage $s$ we build $\B_s$ a finite structure in a finite sub-signature $\mathcal{L}_s$ of the signature $\mathcal L$ of $\N$. We copy  elements from $\N$, attempting to only copy the elements in $\M$. At each stage $s$, we have an embedding $f_s$ from our $\B_s$ into $\N$.
	 We start with $\B_0=\emptyset$, $\mathcal{L}_s=\emptyset$, and $f_0=\emptyset$.
	
	We now describe the stage $s+1$-construction:
	If the range of $f_s$ is contained in $M_{s+1}$, then we take $y$ to be the least member of $M_{s+1}$ which is not in the range of $f_s$, and we append a new element $x$ to $\B_{s+1}$ to copy this element. We extend $f_s$ to set  $f_{s+1}(x)=y$. If $\mathcal{L}_s\neq \mathcal L$, we add another symbol from $\mathcal L$ to $\mathcal{L}_s$ to form  $\mathcal{L}_{s+1}$. We then determine atomic facts in $\B_{s+1}$ to maintain that $f_{s+1}$ is an embedding of $\B_{s+1}$ into $\N$. 
	
	If the range of $f_s$ is not contained in $M_{s+1}$, then we let $z$ be the least element of the range of $f_s$ which is not in $M_{s+1}$. Let $\bar{c}$ be the elements of $M_{s+1}\cap \ran(f_s)$ and $z\bar{y}$ be the elements of $\ran(f_s)\smallsetminus M_{s+1}$. 	
	 We wait for a stage $t>s$ such that one of two things happens: Either (Outcome 1) $M_t\cap \ran(f_s)\neq M_{s+1}\cap \ran(f_s)$ or (Outcome 2) there is some element $a\in A_t$ and tuple $\bar{y}'$ so that altering $f_s$ to send $f_s^{-1}(z)$ to $a$ instead of $z$ and $f_s^{-1}(\bar{y})$ to $\bar{y}'$ will still give an embedding of $\B_s$ into $\N$. We speed-up the construction by doing nothing at each stage between $s$ and $t$ (so $\B_{t-1}=\B_s$, $f_{t-1}=f_s$, and $\mathcal{L}_{t-1}=\mathcal{L}_s$). In the first case, we simply proceed to the next stage with $f_{t} = f_{s}$. Note that the first outcome can happen only finitely many consecutive times, since the $\Delta^0_2$-approximation must settle down on $\ran(f_s)$. In the second case, we let $f_{t}$ be as described and go to the next stage. In both cases, we keep $\B_t=\B_s$ and $\mathcal{L}_t=\mathcal{L}_s$.
	 
	 We first argue that one of these two outcomes must happen. Suppose otherwise. In particular, $M_{s+1}\cap \ran(f_s)=M_t\cap \ran(f_s)$ for all $t>s$, so $\bar{c}\subseteq \M$. Then consider the formula $\exists \bar{y}\phi(\bar{c},z,\bar{y})$ where $\phi$ is the formula describing all of our atomic and negation of atomic commitments regarding the finite structure $\B_s$. By strong minimality, this describes either a finite or co-finite subset of $\N$. First suppose that it is finite, so $z\in \acl(\bar{c})$. 
	 But since $\bar{c}\subseteq \M$, this shows that $z\in \M$. But $z\notin M_{s+1}$. Thus the approximation to $\M$ on $\ran(f_s)$ must change at some point yielding Outcome 1 and a contradiction. Next suppose that it is co-finite. Then there is some $a\in A$ satisfying this formula. This exactly gives us the $a\bar{y}'$ needed to attain Outcome 2, again yielding a contradiction.
	 	
	We now show that for every $x\in \B$, $\lim_s f_s(x)$ converges to a member of $\M$. This is shown by induction. Suppose that all previous elements of $\B$ have already had $f_s$ converge to elements of $\M$. Then $f_t(x)$ can only change after stage $s$ by moving to become an element of $A_t$. This assignment is then permanent since $x\in M_r$ for all later $r$, due to $x$ being in $A_r$. Further, if $f_s(x)$ is not a member of $\M$, then at some stage $t$ where the approximation to $\M$ has settled down on $x$, perhaps after finitely many Outcome 1s, we will reassign $f$ on $x$ to be a member of $A$. 
	Thus $\lim_s f_s(x)$ exists and is a member of $\M$.

	Finally, we argue that every member $x\in \M$ is in the range of $\lim_s f_s$. This is also proved by induction. Let $s$ be a stage so that all previous elements of $M$ are in the range of $f_s$ and that $\lim f_s$ has already converged on these elements. Further, let $s$ be late enough that the approximation to $\M$ has settled on $x$. Since we only change $f_t$ on elements whose images do not appear to be in $\M$ at stage $t$, if $x\in \ran(f_t)$, for any $t>s$, then $x\in \ran(\lim_s f_s)$. Since $\lim_s f_s$ is well defined on $\B_s$ and $\ran(\lim_s f_s)\subseteq \M$, we can take some stage $t\geq s$ where $\ran(f_t)\subseteq M_{t+1}$. Then $x\in \ran(f_{t+1})$, so $x\in \ran(\lim_s f_s)$.
	Thus $\lim_s f_s$ gives an isomorphism from $\B$ to $\M$ showing that $\M$ has a recursive copy.
\end{proof}

Next we see that in any model complete strongly minimal theory, the $k$-dimensional model is $\Delta^0_2$ in the $k+1$-dimensional model. With the previous lemma, this focuses the problem on finding $\Sigma^0_1$ subsets of models.

\begin{lemma}\label{FDClosedSetsAreDelta2}
	Let $T$ be a strongly minimal model complete theory. Let $\N$ be a finite-dimensional model and $\bar{b}$ a generic tuple (possibly empty) in $\N$. Then $\acl(\bar{b})$ is a $\Delta^0_2$ subset of $\N$.
\end{lemma}
\begin{proof}
	Note that for any tuple $\bar{a}$, $d\in \acl(\bar{a})$ if and only if there exists an existential formula $\phi$ and natural number $n$ so that $\N\models \phi(\bar{a},d)\wedge \neg \exists^{n} x \phi(\bar{a},x)$. By model completeness of $T$, our restriction to existential formulas gives the full algebraic closure. This shows that $\{d,\bar{a}\mid d\in \acl(\bar{a})\}$ is a $\Sigma^0_2$ subset of $\N^{<\omega}$. 
	Thus for any tuple $\bar{a}$, $\acl(\bar{a})$ is $\Sigma^0_2$. 
	
	We extend $\bar{b}$ to a basis $\bar{b},c_1,c_2\ldots,c_k$ of $\N$. By repeatedly using the exchange property, $x\in \acl(\bar{b})$ if and only if $c_k\notin \acl(\bar{b},x,c_1,\ldots c_{k-1})\wedge \cdots \wedge c_1\notin \acl(\bar{b},x)$. As this is a conjunction of $\Pi^0_2$-conditions, $\acl(\bar{b})$ is also $\Pi^0_2$.
\end{proof}

Due to the previous two Lemmas, we now turn towards finding infinite $\Sigma^0_1$ subsets of models of model complete flat strongly minimal theories. This will rely heavily on the geometric assumption of flatness. Towards this end, we introduce the purely geometric construction of ping-pong sequences. We then show that ping-pong sequences can be found inside the algebraic closure of a generic tuple via a single formula. The core use of geometry is in showing that the ping-pong sequences must be injective, thus giving an infinite set. This will give us infinite $\Sigma^0_1$ sets. 

We note that this strategy hinges on a strong form of non-local-finiteness. In particular, we will have a single formula $\phi$ so that, starting from a finite tuple, repeatedly using $\phi$ to generate algebraicities will give us an infinite set. Note that there are disintegrated strongly minimal theories where this cannot be done even with algebraic closures being infinite, such as the constructions employed in \cite{KNS97,Ni99,HKS06}. Similarly, there are locally modular strongly minimal theories where this cannot be done, such as the theory of a vector space over $\mathbb{F}_5^{alg}$. Thus, this strategy requires the geometric assumption of flatness.

\section{Infinite formula-closures}\label{sec:PPS}

\subsection{Ping pong sequences}

In the context of a non-disintegrated pregeometry $G$, we describe a procedure for generating a sequence $(t_i : i<n)$, for some $n\in [0,\omega]$. We do this so each $t_{i+1}$ is interalgebraic with $t_i$ over either $X\cup \{a_1\}$ or $X\cup \{a_2\}$ depending on the parity of $i$. We think of $X$ as a ``net'' and the elements $a_1,a_2$ as paddles with which we ``hit'' $t_i$ to find $t_{i+1}$ (see Figure \ref{pingpongs}).

Choose arbitrarily a set $X\subseteq G$ and points $a_1,a_2\in G$ independent over $X$. For every $Y\subseteq G$ denote $\cl_X(Y) = \cl(X\cup Y)$.
Choose arbitrarily $t_1\notin \cl_X(a_1,a_2)$. For each $i\in\omega$, let $j_i\in \set{1,2}$ be such that $i\parity j_i \mt$. 
If possible, choose $t_{i+1}\in\cl_X(a_{j_i},t_i)\setminus \cl_X(a_{j_i})$ distinct from $t_i$. Call such a sequence a \emph{ping-pong sequence}, or PPS for short.

\begin{figure}[H]
\centering
\includegraphics[width=\linewidth]{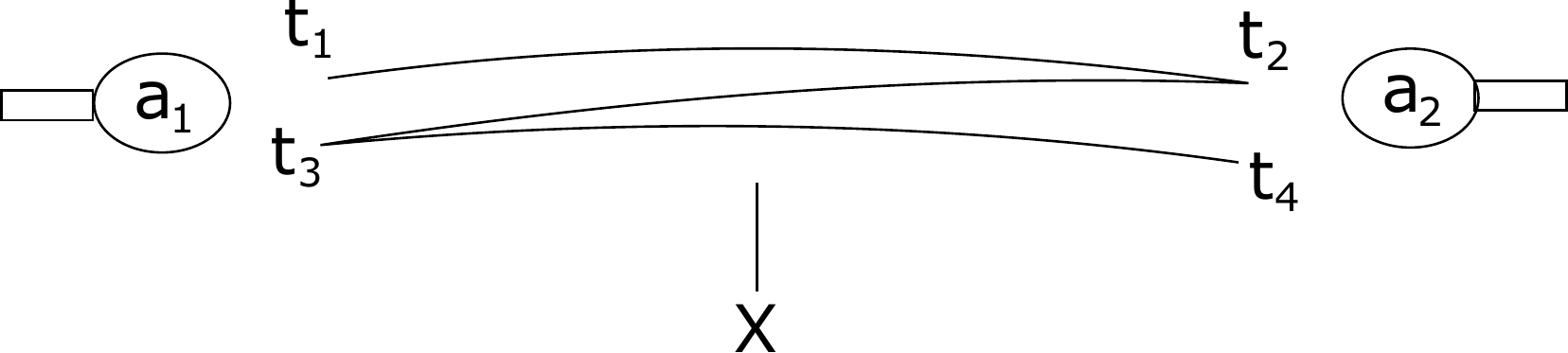}
\caption{}
\label{pingpongs}
\end{figure}

\begin{observation}
	For each $i\in \omega$, if $t_{i+1}$ exists, then $t_{i+1}\notin \cl_X(a_1,a_2)$.
\end{observation}
\begin{proof}
	For $i=1$, $t_1$ was chosen so that $t_1\notin \cl_X(a_1,a_2)$. Assume for a moment there is some $i$ such that $t_i\notin \cl_X(a_1,a_2)$, but $t_{i+1}\in \cl_X(a_1,a_2)$. By construction, using the exchange property, $t_i\in\cl_X(a_{j_i},t_{i+1})\subseteq  \cl_X(a_1,a_2)$, which gives a contradiction. So $t_i\notin \cl_X(a_1,a_2)$ for every $i\in\omega$.
\end{proof}

\begin{lemma}
\label{lemma: PPS is injective}
In a flat pregeometry $G$, a PPS is injective. That is, if $t_{i_1}, t_{i_2}$ are elements of a PPS with $i_1\neq i_2$, then $t_{i_1}\neq t_{i_2}$.
\end{lemma}

\begin{proof}
A consecutive subsequence of a PPS is also a PPS, so it is enough to prove the statement for $i_1=1$. Assume to the contrary that the statement is false. Let $l$ be minimal such that there exists a PPS $t_1,\dots,t_{l+1}$ with $t_1=t_{l+1}$. Fix a PPS $t_1,\dots, t_{l+1}$ generated by $X$, $a_1$, $a_2$ in a flat pregeometry $G$ witnessing this.

Fix $X_0\subseteq X$ a basis for $X$. For each $i< l+1$, let $F_i = X_0\cup\set{a_{j_i}}$ and let $E_i=\cl_X(a_{j_i}t_i) = \cl(F_i\cup\set{t_i})$. Note that also $E_i=\cl_X(a_{j_i}t_{i+1})$ since $t_i$ and $t_{i+1}$ are interalgebraic over $X\cup \{a_{j_i}\}$.

Let $\Sigma = \setcol{E_i}{i<l+1}$. We will find an upper bound on $\dm(\bigcup \Sigma)$ using flatness of $G$. For this, we need to understand the dimensions of intersections of the form $\bigcap_{i\in s} E_i$, for $s$ a set of indices.

\medskip
\noindent\textbf{Claim.} For a set of indices $s\subseteq\set{1,\dots, l}$ with $|s|\geq 2$ and not of the form $\set{i,i+1}$ or $\set{1,l}$,
\[
\bigcap_{i\in s} E_i = \cl\left(\bigcap_{i\in s} F_i\right)=
\begin{cases}
\cl(X\cup\set{a_{j}}) \text{ if }j\in \set{1,2}\text{ and } \forall i\in s\, i \parity a_j \mt \\
\cl(X) \text{ if there are }i,j\in s\text{ with }i\not\parity j \mt
\end{cases}
\]

\begin{proof}
Fix some $1\leq m<k< l+1$. We first show that $E_m\cap E_k \subseteq  \cl_X(a_1,a_2)$ unless $\set{k,m}$ equals $\set{i,i+1}$ or $\set{1,l}$.
Towards the contrapositive, let $E_m\cap E_k \nsubseteq  \cl_X(a_1,a_2)$. We will show that either $k=m+1$, or $m=1$ and $k=l$. Let $t\in {E_k\cap E_m\setminus \cl_X(a_1,a_2)}$. If $t=t_m$, then $t_m,\dots, t_k, t_m$ is a PPS, contradicting minimality of $l$ unless $m=1$ and $k=l$, so we may assume $t\neq t_m$. We split into cases, depending on whether $m$ and $k$ have the same parity. Without loss of generality, assume $m\parity 1 \mt$.

\noindent\underline{Case $k\parity 1 \mt$}: Then $t\in E_k = \cl_X(a_1,t_{k+1})$, implying $t_{k+1}\in\cl_X(a_1,t)\subseteq E_m$. Unless $t_{k+1} = t_m$, which by minimality of $l$ can only happen if $m=1$ and $k=l$, the sequence $t_1,\dots, t_m, t_{k+1},\dots, t_{l+1}$ is a PPS contradicting the minimality of $l$.

\noindent\underline{Case $k\parity 0 \mt$}: We may assume $t\neq t_{k+1}$, for otherwise by exchange $E_m = \cl_X({a_1,t}) = E_{k+1}$ and $t_{k+1}, t_{m+1},\dots t_k, t_{k+1}$ is a PPS, which by minimality of $l$ implies $m=1$ and $k=l$.
Since $t\in E_k = \cl_X({a_2,t_{k+1}})$ implies ${t_{k+1}\in \cl_X({a_2,t})}$, as $t\in E_m$ and $t\neq t_{k+1}$, the sequence $t_1,\dots, t_m, t, t_{k+1},\dots, t_{l+1}$ is a PPS. By minimality of $l$, this means that $k=m+1$.

\smallskip
By what we've shown, for a set $s\subseteq \set{1,\dots, l}$ of size at least $2$ that is not of the form $\set{i,i+1}$ or $\set{1,l}$, there are $m,k\in s$ such that $E_m\cap E_k \subseteq  \cl_X(a_1,a_2)$. Thus $\bigcap_{i\in s}E_i = \cl_X(a_1,a_2)\cap \bigcap_{i\in s}E_i = \bigcap_{i\in s}  \cl_X(a_1,a_2)\cap E_i $.

For every $i$, as $\cl_X({a_{j_i}}) \subseteq E_i\cap \cl_X({a_1,a_2})\subsetneq E_i$ and $\dm(E_i)=\dm(X\cup\set{a_{j_i}})+1$, it must be that $E_i\cap \cl_X({a_1,a_2}) = \cl_X({a_{j_i}})$. Thus, taking $s$ as in the claim we have
\[
\bigcap_{i\in s} E_i = \bigcap_{i\in s} \cl_X(a_{j_i}) = \cl_X\left(\bigcap_{i\in s}\set{a_{j_i}}\right) = \cl\left(\bigcap_{i\in s} F_i\right).
\]
The penultimate equality is by $a_1$ and $a_2$ being independent over $X$.
\end{proof}

We now need to address sets of indices $s$ that do not fall under the claim. For $s=\set{i}$, where $i<l+1$, we already know $\dm(E_i) = |F_i|+1$.

Consider $s=\set{i,i+1}$ for some $i< l$. Because $t_{i+1},a_1,a_2$ are independent over $X$, we know $a_{j_i}\notin E_{i+1}$, so in particular $E_i\cap E_{i+1}\neq E_i$. Now, ${X\cup \set{t_{i+1}}\subseteq E_i\cap E_{i+1}\subsetneq E_i}$ and $\dm(E_i) = \dm(X\cup\set{t_{i+1}})+1$, hence $E_i\cap E_{i+1} = \cl(X\cup\set{t_{i+1}})$. Thus, $\dm(E_i\cap E_{i+1}) = \dm(X) + 1 = |F_i\cap F_{i+1}|+1$

Consider $s=\set{1,l}$. If $l\parity 2 \mt$, because $t_{l+1} = t_1$, then $E_l = \cl_X({a_2,t_l}) = \cl_X({a_2,t_{l+1}}) = \cl_X({a_2,t_1})$ and $E_1 = \cl_X({a_1,t_1})$, so $\dm(E_l\cap E_1)= |F_l\cap F_1|+1$  is proved exactly as the case $s=\set{i,i+1}$. 
If $l\parity 1 \mt$, then $E_1 = E_l = \cl_X({a_1,t_1})$ and $F_1 = F_l$, so again $\dm(E_1\cap E_l) = \dm(E_1) = |F_1| + 1 = |F_1\cap F_l| + 1$.

Due to the above, we will see some cancellation in the alternating sum $\Delta(\Sigma)$, which we compute as the bound to $d(\bigcup_{i=1}^l E_i)$. For each $i< l$ we have $\dm(E_i) - {\dm(E_i\cap E_{i+1})} = |F_i| - |F_i\cap F_{i+1}|$, and for $i=l$, we have $\dm(E_l) - \dm(E_1\cap E_l) = |F_l| - |F_1\cap F_l|$. Therefore, using this observation and the claim, and then the inclusion-exclusion principle,
\begin{align*}
\Delta(\Sigma)=
\sum_{\emptyset\neq s\subseteq\set{1,\dots, l}} (-1)^{|s|+1} \dm\left(\bigcap_{i\in s} E_i\right) = \sum_{\emptyset\neq s\subseteq\set{1,\dots, l}} (-1)^{|s|+1} \left|\bigcap_{i\in s} F_i\right| = \left|\bigcup_{i=1}^l F_i\right|.
\end{align*}
But $\bigcup_{i=1}^l F_i = X_0\cup\set{a_1,a_2}$, so by flatness, $\dm\left(\bigcup_{i=1}^l E_i\right) \leq \left|X_0\cup\set{a_1,a_2}\right| = \dm(X\cup\set{a_1,a_2})$. Since $t_1\in \bigcup_{i=1}^l E_i$, this is a contradiction to $t_1\notin\cl_X({a_1,a_2})$, which proves the lemma.
\end{proof}

\subsection{Formula-closures}

We define the formula-closure of a set.

\begin{definition}
	Let $\mathcal{N}$ be a model, $X\subseteq \mathcal{N}$ a set, and $\psi(x_1,\dots, x_m)$ some formula. Let $X^0_{\psi} = X$ and recursively define 
	\[
	X^{i+1}_{\psi} = \bigcup_{\substack{a_1,\dots,a_{m-1}\in X^i_{\psi} \\ 0\leq j< m}} \psi(a_1,\dots, a_j,\mathcal{N},a_{j+1},\ldots,a_{m-1}).
	\]
	We define the $\psi$-closure of $X$ in $\mathcal{N}$ to be $\Lambda^{\mathcal{N}}_{\psi}(X) = \bigcup_{i\in\omega} X^i_{\psi}$.
\end{definition}

We now turn to choosing a formula $\phi$ so that $\Lambda^\N_\phi$ is algebraic and, in many cases, will be infinite. 

\begin{definition}
	In a pregeometry, a \emph{cirucit} is a dependent tuple whose every proper subtuple is independent.
\end{definition}

\begin{definition}\label{def:phi}
	For $T$ a strongly minimal theory whose pregeometry is non-disintegrated, we let $\phi_0(x_1,\ldots x_m)$ be a formula which witnesses the smallest circuit of size $>2$ in the pregeometry of $T$, i.e., $\varphi_0$ holds on some circuit of size $m$ in the countable saturated model of $T$, but not on a generic $m$-tuple. The existence of $\varphi_0$ follows from non-disintegration.
	
	We let $\phi$ be a formula which witnesses the smallest circuit of size $>2$ in the pregeometry of $T$ so that $\phi$ makes each coordinate algebraic over the others. That is, for each $1\leq j\leq m$, $T\models \forall \bar{x}\exists^{<\infty}y \phi(x_1,\dots, x_{j-1},y,x_{j+1},\ldots,x_m)$. Note that choosing $\phi(\bar{x})$ to be $\phi_0(\bar{x})\wedge \bigwedge_{j\leq m}\exists^{<\infty}y \phi_0(x_1,\dots, x_{j-1},y,x_{j+1},\ldots,x_m)$ suffices, as every circuit satisfying $\varphi_0$ also satisfies $\varphi$. That our choice of $\varphi$ can be expressed as a first order formula follows from strong minimality of $T$.
	
	We define $n$ to be $m-1$. This is the dimension of the smallest circuit of size $>2$, which will be more natural to use below than the size of the circuit.
\end{definition}

The next observation and lemma are in terms of $T$, $\varphi$, $n$ as in Definition \ref{def:phi}.

\begin{observation}\label{LambdaIsAlgebraic}
	By choice of $\varphi$, since $\varphi(a_1,\dots, a_{i-1},\mathcal{N},a_{i+1},\dots,a_m)$ is always finite in $\mathcal{N}\models T$, for every $X\subseteq \mathcal{N}$ we have $\Lambda^{\mathcal{N}}_{\varphi}(X)\subseteq \acl_{\mathcal{N}}(X)$.
\end{observation}

\begin{lemma}
	\label{lemma: infinite PPS over indepednent n+1-tuple}
	Let $\N\models T$ and suppose that the pregeometry of $T$ is flat.
	 If $Y\subseteq \mathcal{N}$ contains $n+1$ independent points, then $\Lambda^{\mathcal{N}}_{\varphi}(Y)$ is infinite.
\end{lemma}

\begin{proof}
	Let $t_1,a_1,a_2,y_1,\dots,y_{n-2}\in Y$ be independent. Beginning with $t_1$ and using $a_1$, $a_2$, and $X = \set{y_1,\dots, y_{n-2}}$, we define a PPS contained in $\Lambda^{\mathcal{N}}_{\varphi}(Y)$. Assuming we were able to construct $t_1,\dots, t_i\notin \acl(X\cup\set{a_1,a_2})$, we only need to show that an appropriate $t_{i+1}$ exists.
	
	Since $T$ is strongly minimal, every independent $n$-tuple in $\mathcal{N}$ has the same type. In particular, by choice of $\varphi$, the independent $n$-tuple $(t_i,a_{j_i},X)$ can be extended to a circuit which is an instance of $\varphi$ by an element $t_{i+1}$. By virtue of $(t_i,a_{j_i}, X, t_{i+1})$ being a circuit, $t_{i+1}\in \acl({X}\cup\set{a_{j_i},t_i})\smallsetminus \acl(X\cup\set{a_{j_i}})$.
	
	We get a PPS $(t_i : i\in\omega)$, which by Lemma \ref{lemma: PPS is injective} contains infinitely many elements. By construction, clearly $t_i\in \Lambda^{\mathcal{N}}_{\varphi}(Y)$ for every $i\in\omega$.
\end{proof}

\section{Recursion theoretic consequences}\label{sec5}

Fix $T$ a model complete flat strongly minimal theory. We let $\phi$ and $n$ be as defined in Definition \ref{def:phi}. Suppose that $\N$ is a finite-dimensional recursive model of $T$. In this section, we will show that every model $\M\preceq \N$ of dimension at least $n$ has a recursive presentation. We first show the result for models $\M$ of dimensions at least $n+1$.

\begin{theorem}\label{thm:DownTon+1}
	Suppose that $\N$ is a finite dimensional recursive model of $T$. Then every model of $T$ with dimension $k\in [n+1,\text{dim}(\N)]$ has a recursive presentation. 
 \end{theorem}
\begin{proof}
	Fix a generic $k$-tuple $\bar{x}$ in $\N$. By Lemma \ref{FDClosedSetsAreDelta2}, $\acl(\bar{x})$ is a $\Delta^0_2$ subset of $\N$. By model completeness of $\N$, $\phi$ defines a recursive subset of $\N^{n+1}$, as both it and its negation are existentially defined by model completeness. It follows that $\Lambda_\phi^\N(\bar{x})$ is a $\Sigma^0_1$ subset of $\N$. By Lemma \ref{lemma: infinite PPS over indepednent n+1-tuple}, this $\Sigma^0_1$-set is infinite. It follows from Lemma \ref{GoingDownTrick} that $\acl(\bar{x})$, which is isomorphic to the $k$ dimensional model of $T$, has a recursive presentation.
\end{proof}

Our next goal is to extend this result also to the $n$-dimensional model of $T$. 
For what follows, it is convenient to note that not only is $\Lambda_\phi^{\N}(X)$ recursively enumerable for a given finite $X$, but the sequence $(X_\phi^i)_{i\in \omega}$ is uniformly recursive in $X$.

\begin{lemma}\label{FinLambdaIsRE}
	The set of triples $(y,X,i)$ so that $y\in X_\phi^i$, where $X$ is a finite set given by canonical index, is recursive. It follows that the set of $X$ for which $\Lambda_\phi^{\N}(X)$ is finite is a recursively enumerable set.
\end{lemma}
\begin{proof}
	\sloppy
	Since $\phi$ makes each coordinate algebraic over the others, that is  $T\models \forall \bar{x}\exists^{<\infty}y \phi(x_1,\dots, x_{j-1},y,x_{j+1},\ldots,x_m)$, there is a uniform bound $K$ so that $T\models \forall \bar{x}\exists^{<K}y \phi(x_1,\dots, x_{j-1},y,x_{j+1},\ldots,x_m)$. By model completeness, each formula $\exists^{=L}y\phi(x_1,\dots, x_{j-1},y,x_{j+1},\ldots,x_m)$ is a recursive set. Thus, from the finite set $X_\phi^i$, we can determine for each tuple $\bar{x}$ exactly how many $y$ we need to find to include in $X_\phi^{i+1}$. Thus it is recursive to find all such $y$, and thus to find $X_\phi^{i+1}$. For the second statement, note that $\Lambda_\phi^{\N}(X)$ is finite if and only if there exists an $i$ so that $X_\phi^i=X_\phi^{i+1}$, which is a $\Sigma^0_1$ condition.
\end{proof}

\begin{lemma}\label{DimensionnHaveSigma1subsets}
	Assume $\N$ is a recursive model of $T$ of dimension $\geq n+1$. Let $X\subseteq \N$ be an algebraically closed subset of dimension $n$. Then if $X$ is infinite, it contains an infinite $\Sigma^0_1$ set.
\end{lemma}

\begin{proof}
	If there is some finite $B\subseteq X$ with $\Lambda^{\N}_{\varphi}(B)$ infinite, this is the needed infinite   $\Sigma^0_1$ set, so we may assume there is no such $B$.
	
	Fix $\bar{b}$ to be a basis for $X$. By Lemma \ref{lemma: infinite PPS over indepednent n+1-tuple}, $\Lambda_\phi^{\N}(\bar{b},a)$ is infinite for any $a\notin X$. However, for any $a\in X$, $\Lambda_\phi^{\N}(\bar{b},a)$ is finite. Thus Lemma \ref{FinLambdaIsRE} shows that $X$ itself is $\Sigma^0_1$ as it is the set of $a$ so that $\Lambda_\phi^{\N}(\bar{b},a)$ is finite.
\end{proof}

We can now extend our result to the dimension $n$ model as well.

\begin{theorem}\label{thm:downton}
	Suppose that $\N$ is a finite dimensional recursive model of $T$. Then every model of $T$ with dimension $k\in [n,\text{dim}(\N)]$ has a recursive presentation. 
\end{theorem}
\begin{proof}
	For $k\geq n+1$, this was shown in Theorem \ref{thm:DownTon+1}. For the case of dimension $n$, we again fix $\bar{x}$ a generic $n$-tuple in $\N$ and Lemma \ref{FDClosedSetsAreDelta2} shows that $\acl(\bar{x})$ is $\Delta^0_2$. Since there is assumed to be a model of dimension $n$, $\acl(\bar{x})$ is infinite and Lemma \ref{DimensionnHaveSigma1subsets} shows that it contains an infinite $\Sigma^0_1$ set. Then Lemma \ref{GoingDownTrick} shows that $\acl(\bar{x})$, which is isomorphic to the $n$-dimensional model of $T$, has a recursive presentation.
\end{proof}

\section{Relating circuit size to model dimensions}\label{sec:relatingcircuitsizetomodeldimensions}

Thus far we have shown that for a model complete flat strongly minimal theory $T$, if a finite dimensional model is recursive, then all the models of smaller dimension, down to $n$ (the dimension of the smallest circuit) are also recursively presentable. We now consider how close this gets us to showing that all models of smaller dimension are recursively presentable. To do this, we let $p$ be the dimension of the prime model of $T$ and we will give some relationships between $n$ and $p$. Recall that we have fixed a formula $\phi$ at the beginning of section \ref{sec5} as in Definition \ref{def:phi}.

\begin{lemma}
\label{lemma: p leq n plus 1}
$p\leq n+1$
\end{lemma}

\begin{proof}
	By Lemma \ref{lemma: infinite PPS over indepednent n+1-tuple} and Observation \ref{LambdaIsAlgebraic}, the closure of a generic $n+1$-tuple is infinite. Thus Fact \ref{TVFact} shows that there is an $n+1$-dimensional model of $T$.
\end{proof}

\begin{observation}
\label{observation: if disintegrated model exists, the closure of every generic point is infinite}
If $\mathcal{N}\models T$ is of dimension less than $n$, then
$G_{\mathcal{N}}$ is disintegrated. In particular, if such a model $\mathcal{N}$ exists, by the pigeonhole principle, the closure of every generic point in a model of $T$ is infinite.
\end{observation}

\begin{lemma}\label{lem:pgeqnforlargen}
If $n>3$, then $p\geq n$.
\end{lemma}

\begin{proof}
Assume to the contrary that $n>3$, but $p<n$. Let $\mathcal{M}_{n-1}$ be the model of $T$ of dimension $n-1$, which exists by Fact \ref{TVFact}, and let $b_1,\dots,b_{n-1}$ be a basis for $\M_{n-1}$. By Observation \ref{observation: if disintegrated model exists, the closure of every generic point is infinite}, the pregeometry $G_{\mathcal{M}_{n-1}}$ is disintegrated, and the closure of every point in $\mathcal{M}_{n-1}$ is infinite. In particular, in $\acl(b_{n-1})$ there is a generic enough point $c$ such that
\[
\mathcal{M}_{n-1}\models \exists x \varphi(b_1,\dots,b_{n-1},c,x).
\]
Choose such a $c$, and $d\in \varphi(b_1,\dots, b_{n-1},c,\mathcal{M}_{n-1})$. Since $d\in\acl(b_1,\dots, b_{n-1},c) = \acl(b_1,\dots, b_{n-1})$, by disintegration and independence of $b_1,\dots, b_{n-1}$, either $d\in\acl(\emptyset)$ or there is a unique $i$ such that $d\in\acl(b_i)$. In particular, because $n>3$, there is some $i'\neq n-1$ such that $d\in\acl(b_1,\dots,b_{i'-1},b_{i'+1},\dots, b_{n-1})$. Without loss of generality, $i'=1$.

By choice of $\varphi$ in Definition \ref{def:phi}, we know $\varphi(\mathcal{M}_{n-1}, b_2,\dots, b_{n-1},c,d)$ is finite, so $b_1$ is not generic over $b_2,\dots, b_{n-1},c,d$. But $c,d\in\acl(b_2,\dots,b_{n-1})$, so $b_1\in \acl(b_2,\dots, b_{n-1})$, in contradiction to $b_1,\dots, b_{n-1}$ being a basis.
\end{proof}

\begin{remark}
The proof of Lemma \ref{lem:pgeqnforlargen} can be modified slightly to show that if $n=3$, then $p\geq 1$. Observe that if $p=0$, then $c$ can be chosen from $\acl(\emptyset)$. In that case, the proof goes through also for $i'=n-1$, because unlike before, $c\in\acl(b_1,\dots,b_{i'-1},b_{i'+1},\dots, b_{n-1})$ regardless of choice of $i'$. Thus, we only need $n\geq 3$ to get a contradiction, implying that if $n=3$, then $p> 0$.
\end{remark}

We now extend Lemma \ref{lem:pgeqnforlargen} to the case $n=3$. Note that in this Lemma, we use flatness in an essential way. We do not know if the result holds without the assumption of flatness.

\begin{lemma}\label{lem:pgeqnforn3}
	If $n=3$, then $p\geq n$.
\end{lemma}
\begin{proof}
	We fix $n=3$.
	Assume, towards a condradiction, that the model of dimension 2 exists. The type of an independent tuple in a strongly minimal theory is not dependent on the dimension of the model, so by Observation \ref{observation: if disintegrated model exists, the closure of every generic point is infinite}, the closure of every independent point is infinite, and $\acl(a,b) = {\acl(a)\cup \acl(b)}$ for any independent pair $a,b$.
	
	Let $\psi(x_1,\dots,x_6) := \varphi(x_1,x_2,x_5,x_6)\land \varphi(x_3,x_4,x_5,x_6)$.
	
	\smallskip
	\sloppy
	\noindent\textbf{Claim 1}.
	If $\set{a,b,c,d}$ is independent in some $\M\models T$, then $\mathcal{M}\models \exists x_5,x_6 \psi(a,b,c,d,x_5,x_6)$.
	
	\begin{proof}
		By homogeneity, whenever $x,y,z$ are independent there exists $w\in\varphi(\mathcal{M},x,y,z)$ such that $\set{w,x,y,z}$ is a circuit. Take $d,e,f$ independent, and let $c$ be such that $\mathcal{M}\models \varphi(c,d,e,f)$ and $\set{c,d,e,f}$ is a circuit. Now take $b$ independent from $c,d,e,f$ and let $a$ be such that $\mathcal{M}\models \varphi(a,b,e,f)$ and $\set{a,b,e,f}$ is a circuit.
		
		We claim that $\set{a,b,c,d}$ are independent. Assume not, i.e., $a\in\acl(b,c,d)$. By $n=3$, either $a$ is interalgebraic with a single element from $\set{b,c,d}$, or $\set{a,b,c,d}$ is a circuit. By construction, $a\notin \acl(b)$. If $a\in \acl(c,d)\subseteq \acl(c,d,e,f)$, then $b\in\acl(a,e,f)\subseteq\acl(c,d,e,f)$, so this is also not the case. Thus $\set{a,b,c,d}$ is a circuit. Denote $E_1 =\acl(a,b,c,d)$, $E_2 = \acl(a,b,e,f)$, $E_3 = \acl(c,d,e,f)$. The sets $E_1,E_2,E_3$ are distinct -- the union of any two is of dimension $d(a,b,c,d,e,f) = 4$, whereas $\dm(E_i)=3$ for each $i\in\set{1,2,3}$. Thus, $\dm(E_i\cap E_j) = 2$ for any $i,j$ distinct. As for $\dm(E_1\cap E_2\cap E_3)$, note $E_1\cap E_2\cap E_3 = (E_1\cap E_3)\cap (E_2\cap E_3) = \acl(c,d)\cap \acl(e,f) = (\acl(c)\cup\cl(d))\cap (\acl(e)\cup \acl(f)) = \acl(\emptyset)$. So $\dm(E_1\cap E_2\cap E_3) = 0$.
		
		Executing a flatness calculation for $\Sigma = \set{E_1,E_2,E_3}$ yields
		\[
		\dm(abcdef)\leq\Delta(\Sigma)= 3\cdot 3-3\cdot 2 + 0 = 3
		\]
		which contradicts independence of $\set{b,d,e,f}$. Thus, $a,b,c,d$ are independent.
		
		As $\M\models \exists x_5,x_6 \psi(a,b,c,d,x_5,x_6)$ and the type of an independent 4-tuple in a model of $T$ is unique, this proves the claim.
	\end{proof}
	
	Fix $a,b,c$ independent in some $\mathcal{M}\models T$. By Claim 1, $\exists x_5,x_6 \psi(a,b,c,\mathcal{M},x_5,x_6)$ is co-finite, so intersects $\acl(a)$. Let $a'\in\acl(a)$ be such that $\mathcal{M}\models \exists x_5,x_6 \psi(a,b,c,a',x_5,x_6)$. Now $\exists x_5,x_6 \psi(a,b,\mathcal{M}, a',x_5,x_6)$ is co-finite, because $c$ is independent of $a,b$. In particular, we may take $a''\in\acl(a)$ such that $\mathcal{M}\models \exists x_5,x_6 \psi(a,b,a'',a',x_5,x_6)$.
	
	\smallskip
	\noindent\textbf{Claim 2}. There are $t_0,t_1\in\acl(a,b)$ such that $\mathcal{M}\models \psi(a,b,a'',a',t_0,t_1)$.
	
	\begin{proof}
		If there are only finitely many $t_0$ satisfying $\exists x_6 \psi(a,b,a'',a',t_0,x_6)$, then any such $t_0$ is in $\acl(a,b)$. Otherwise, by strong minimality, there are co-finitely many such $t_0$ and we may choose $t_0\in \acl(a,b)$ satisfying $\exists x_6 \psi(a,b,a'',a',t_0,x_6)$. With such a $t_0$, $\psi(a,b,a'',a',t_0,\mathcal{M})\subseteq\phi(a,b,t_0,\M)\subseteq \acl(a,b)$. So, for any $t_1$ satisfying $\psi(a,b,a'',a',t_0,t_1)$, we have $t_0,t_1\in \acl(a,b)$.
	\end{proof}
	
	Take $t_0,t_1$ such as in Claim 2. Assume for a moment $t_i\in \acl(a)$ for some $i\in\set{0,1}$. Then $\varphi(a'',a',t_0,t_1)$ implies $t_{1-i}\in \acl(a)$. Now $\varphi(a,b,t_0,t_1)$ implies $b\in \acl(a)$, in contradiction. Therefore, it must be that $t_0,t_1\notin\acl(a)$, i.e. $t_0,t_1\in\acl(b)$. But now $a\in\acl(b,t_0,t_1)=\acl(b)$, a contradiction.
	
	We have thus shown that there is no $2$-dimensional model. It follows from Fact \ref{TVFact} that $p\geq 3=n$.
\end{proof}

\section{Recursion Theoretic Consequences}

We now restate Theorem \ref{thm:downton} in the case $n\neq 2$:

\begin{theorem}\label{thm:For n not equal to 2}
	Let $T$ be a model complete flat strongly minimal theory. Suppose further that the smallest circuit in a saturated model of $T$ has size $\geq 4$ (i.e. $n\geq 3$). Suppose that $\N$ is a recursive model of finite dimension $k$. Then any model of $T$ of dimension $\leq k$ has a recursive presentation. 
\end{theorem}
\begin{proof}
	Theorem \ref{thm:downton} tells us that any model of dimension $\leq k $ has a recursive presentation, given that its dimension is $\geq n$. But Lemmas \ref{lem:pgeqnforlargen} and \ref{lem:pgeqnforn3} show that there are no models of dimension $<n$.
\end{proof}

We now focus on the remaining case of $n=2$. It follows from Theorem \ref{thm:downton} that if any model of finite dimension $\geq 2$ is recursive, then the $2$-dimensional model is recursive. Yet it is possible that the $0$- and $1$-dimensional models exist. Our methods which exploit the pregeometry to find elements in the closure of a tuple are ill equipped for this case. In particular, for the target models, the pregeometry is disintegrated! 

In this case, we are not able to use the assumption of a recursive 2-dimensional model to prove that the $0$- and $1$-dimensional models have recursive presentations, but we are able to do so assuming the existence of a recursive $3$-dimensional model.

Recall that we have fixed a formula $\phi$ at the beginning of section \ref{sec5} as in Definition \ref{def:phi}.

\begin{definition}
	We define the Infinite-$\Lambda$-Dimension ($\ILD$) of $T$ to be the smallest dimension of a tuple $\bar{x}$ in a model of $T$ so that $\Lambda_\phi(\bar{x})$ is infinite.
\end{definition}

\begin{lemma}\label{closure of ILD-1}
	Let $\N$ be a recursive model of $T$ of dimension $\geq \ILD$, and let $\bar{y}$ be an independent $(\ILD-1)$-tuple in $\N$. Then $\acl(\bar{y})$ is a $\Sigma^0_1$-subset of $\N$.
\end{lemma}
\begin{proof}
	Let $\bar{x}$ be a tuple in $\N$ of dimension $\ILD$ so that $\Lambda_\phi(\bar{x})$ is infinite. Partition $\bar{x}$ as $\bar{x}_0\cup \bar{x}_1$ where $\bar{x}_0$ is a basis for $\bar{x}$ and reorder $\bar{x}$ so $\bar{x}=\bar{x}_0\bar{x}_1$. 
	
	We choose a formula $\psi(\bar{u},\bar{v})$ so that
	
	\begin{enumerate}
		\item \label{psi isolates}$\psi(\bar{x}_0,\bar{x}_1)$ isolates the type of $\bar{x}_1$ over $\bar{x}_0$, 
		\item \label{psi is witnessed}$T\models \forall \bar{z} \exists \bar{w} \psi(\bar{z},\bar{w})$,
		\item \label{psi is algebraic} $T\models \forall \bar{z} \exists^{<\infty} \bar{w} \psi(\bar{z},\bar{w})$.
	\end{enumerate} 

We now verify that such a $\psi$ exists. Since $\bar{x}_1$ is algebraic over $\bar{x}_0$, its type is isolated, so we can choose a formula $\psi_0$ isolating its type. Let $M$ be the number such that $\N\models \exists^{=M}\bar{v} \psi_0(\bar{x}_0,\bar{v})$. We can then let $\psi(\bar{u},\bar{v}) = \left(\exists^{=M}\bar{w} \psi_0(\bar{u},\bar{w})\wedge \psi_0(\bar{u},\bar{v})\right)
\vee \left( \neg \exists^{=M}\bar{w} \psi_0(\bar{u},\bar{w}) \wedge \bigwedge_{i<\vert \bar{v}\vert}v_i=u_0\right)$.

\begin{claim}
	$z\in \acl(\bar{y})$ if and only if $\Lambda_\phi(\bar{y},z,\bar{w})$ is finite for some, equivalently for any, $\bar{w}$ so that $\N\models \psi(\bar{y}z,\bar{w})$
\end{claim}
\begin{proof}
	Suppose $z\in \acl(\bar{y})$. Then for any $\bar{w}$ so that $\psi(\bar{y}z,\bar{w})$, the dimension of $\bar{y}z\bar{w}$ is $\ILD-1$ by (\ref{psi is algebraic}). Thus $\Lambda_\phi(\bar{y},z,\bar{w})$ is finite by the minimality in the definition of $\ILD$.
	
	Suppose $z\notin \acl(\bar{y})$. Then the type of $\bar{y}z$ is the same as the type of $\bar{x}_0$. Then for any $\bar{w}$ so that $\psi(\bar{y}z,\bar{w})$, the type of $\bar{y}z\bar{w}$ is the same as the type of $\bar{x}$ by (\ref{psi isolates}). Thus $\Lambda_\phi(\bar{y},z,\bar{w})$ is infinite.
\end{proof}

We can use this to enumerate $\acl(\bar{y})$. We enumerate $z$ into $\acl(\bar{y})$ if we find some tuple $\bar{w}$ so that $\N\models \psi(\bar{y}z,\bar{w})$ and we see that $\Lambda_\phi(\bar{y},z,\bar{w})$ is finite. By (\ref{psi is witnessed}), and the claim, this is precisely $\acl(\bar{y})$. Since model completeness implies that $\psi$ defines a recursive subset of $\N^{\vert \bar{x} \vert}$, Lemma \ref{FinLambdaIsRE} shows that this is a recursive enumeration of $\acl(\bar{y})$.
\end{proof}

\begin{lemma}\label{Rec Below ILD}
	If $T$ has a recursive model of dimension $\geq \ILD$, then every model of $T$ of dimension $< \ILD$ has a recursive presentation.
\end{lemma}
\begin{proof}
	Let $\N$ be a recursive model of $T$ of dimension $\geq \ILD$ and fix $m<\ILD$ so that $T$ has an $m$-dimensional model.
	 Further fix a generic tuple $a_1,\ldots a_\ILD$ in $\N$. For each $i\in [m+1,\ILD]$, let $\bar{c}_i=a_1,\ldots a_{i-1},a_{i+1},\ldots a_{\ILD}$. By Lemma \ref{closure of ILD-1}, each $\acl(\bar{c}_i)$ is a $\Sigma^0_1$ subset of $\N$, thus $\bigcap_{i=m}^\ILD \acl(\bar{c}_i)$ is a $\Sigma^0_1$-subset of $\N$, but Fact \ref{fact:Carousel} shows that this is exactly $\acl(a_1,\ldots a_m)$, so the $m$-dimensional model is recursively presentable.
\end{proof}

\begin{theorem}\label{thm:from3to0and1}
	Suppose that $T$ is a model complete flat 
	strongly minimal theory. Further suppose that the smallest circuit in a saturated model of $T$ has length $3$ (i.e. $n=2$). Suppose that there is a recursive model of $T$ of finite dimension $\geq 3$. Then the $0$- and $1$-dimensional models of $T$ have recursive presentations, if they exist.
\end{theorem}
\begin{proof}
	We know from Lemma \ref{lemma: infinite PPS over indepednent n+1-tuple} that $\ILD\leq n+1=3$. Thus, for any $m<\ILD$, we know that the $m$-dimensional model of $T$ is recursively presentable, if it exists. 
	
	Fix $\N$ a recursive model of $T$ of finite dimension $\geq 3$. For $m=0$ or $m=1$, we have a $\Delta^0_2$-copy of the $m$-dimensional model in $\N$ by Lemma \ref{FDClosedSetsAreDelta2}. But if $m\geq \ILD$, then it contains an infinite $\Sigma^0_1$-subset, namely $\Lambda_\phi(\bar{x})$ for some $\bar{x}$ of dimension $\ILD$. Then by Lemma \ref{GoingDownTrick}, the $m$-dimensional model has a recursive presentation.
\end{proof}

We now exclude a few more spectra not previously excluded.

\begin{theorem}\label{No Weirdness With omega}
	Let $T$ be a flat model complete strongly minimal theory
	Let $1\leq m\in \omega$ and suppose that $T$ has a recursive $m$-dimensional model and a recursive $\omega$-dimensional model. Then every model of $T$ of dimension $<m$ is recursively presentable. 
\end{theorem}
\begin{proof}
	Since the $\omega$-dimensional model of $T$ is recursive, Lemma \ref{Rec Below ILD} 
	shows that every model of $T$ of dimension $< \ILD$ is recursively presentable. So, suppose $\ILD\leq k<m$. Then we have a $\Delta^0_2$-copy of the $k$-dimensional model in the $m$-dimensional model by Lemma \ref{FDClosedSetsAreDelta2}. But since $k\geq \ILD$, it contains an infinite $\Sigma^0_1$-subset, namely $\Lambda_\phi(\bar{x})$ for some $\bar{x}$ of dimension $\ILD$. Then by Lemma \ref{GoingDownTrick}, the $k$-dimensional model has a recursive presentation.
\end{proof}

We now prove our main theorem.

\begin{theoremOne}
	If $T$ is a flat model complete strongly minimal theory, then $\SRM(T)$ is contained in one of the following schema:
	
	\begin{itemize}
		\item $[0,\alpha)$ for $\alpha\in \omega+2$, $[0,n]\cup \set{\omega}$ for $n\in \omega$, $\{\omega\}$, 
		\item $\{1\}$, $\{2\}$, $\{0,2\}$, $\set{1,2}$

	\end{itemize}

Further, the sets in the first line are in fact spectra of flat
model complete strongly minimal theories. We do not know whether any of the 4 sets (not schema) in the second line are spectra of flat model complete strongly minimal theories.
\end{theoremOne}
\begin{proof}
	In the case where $n\neq 2$, Theorem \ref{thm:For n not equal to 2} shows that $\SRM(T) \cap [0,\omega)$ is an initial segment of $[0,\omega)$. Thus $\SRM(T)$ is either of the form $[0,\alpha)$ for some $\alpha\in \omega+2$ or $[0,n]\cup \set{\omega}$ for some $n\in \omega$ or $\set{\omega}$.
	 Next, we consider the case where $n=2$. By Theorem \ref{thm:downton}, the collection of dimensions of recursive models of $T$ is initial among finite dimensions $\geq 2$. Let us first assume that $\SRM(T)\cap [3,\omega)\neq \emptyset$. Then in particular, there is a recursive model of finite dimension $\geq 3$. Then Theorem \ref{thm:from3to0and1} shows that the $0$ and $1$-dimensional models (if they exist) are recursive as well. Thus the collection of dimensions of recursive models of $T$ is initial among finite dimensions. Thus again $\SRM(T)$ is either of the form $[0,\alpha)$ for some $\alpha\in \omega+2$ or $[0,n]\cup \set{\omega}$ for some $n\in \omega$.
	
	Now we consider the case where $n=2$ and $[3,\omega)\cap \SRM(T)=\emptyset$. There are 16 such subsets of $[0,\omega]$. Of these, $\emptyset$, $\set{0}$, $\set{0,1}$, $\set{0,\omega}$, $\set{0,1,2}$, $\set{0,1,\omega}$, $\set{0,1,2,\omega}$, and $\set{\omega}$ are already covered by schemata in the first item. The remaining 8 are the sets in the second item along with $\set{1,\omega}$, $\set{2,\omega}$, $\set{0,2,\omega}$, $\set{1,2,\omega}$. Each of these last four are excluded by Theorem \ref{No Weirdness With omega}.

	The fact that the sets in the schema $[0,\alpha)$ for $\alpha\in \omega+2$ are spectra of strongly minimal theories with finite signatures is proved in \cite{A0n}. In that paper, it is not established that the theory is model complete and has a flat pregeometry, but the construction fits into the framework in \cite[Section 2]{AndrewsMermelstein} and \cite[Lemma 2.23 and Corollary 2.24]{AndrewsMermelstein} show that the theory is model complete and has a flat pregeometry.

	Similarly, the fact that the set $\{\omega\}$ is a spectrum of a strongly minimal theory with a finite signature is proved in \cite{An11}. Once again, this construction fits into the framework in \cite[Section 2]{AndrewsMermelstein}\footnote{There is a slight difference between the papers regarding the definition of an extension being ``of the form'' of another extension, but this makes no difference. In fact, the construction in \cite{An11} can be altered to use the definition from \cite{AndrewsMermelstein} with no change to the rest of the proof.} and so the theory of $\M\vert_{L'}$ (see \cite[Lemma 27]{An11}) is flat and model complete. From there, the model $\M'$ is formed by restricting $\M\vert_{L'}$ to the signature $\{R\}$. Each other symbol from $L'$ is existentially definable in this restriction by  \cite[Lemma 23]{An11}. They are also universally definable because, for every $x,y$ there must be exactly two $z$ so that $R_i(x,y,z)$ holds (since $\mu$ allows this many and $xy\leq \M$). Thus the reduct to $\M'$ is also flat and model complete. The theory of $\M'$ is shown to have spectrum $\{\omega\}$ in \cite[Theorem 30]{An11}.

	Finally, the sets in the schema $[0,n]\cup \set{\omega}$ for $n\in \omega$ are shown to be spectra of flat model complete strongly minimal theories in \cite{AndrewsMermelstein}.
\end{proof}

\bibliographystyle{alpha}
\bibliography{refs}

\end{document}